\documentclass[12pt,a4]{amsart}
\usepackage{amsmath,amssymb,amsthm,comment}
\usepackage{graphicx}
\theoremstyle{plain}
\newtheorem{thm}{Theorem}[section]
\newtheorem{lem}[thm]{Lemma}
\newtheorem{cor}[thm]{Corollary}
\newtheorem{prop}[thm]{Proposition}

\newtheorem{fact}[thm]{Fact}

\newtheorem{rem}[thm]{Remark}

\newcommand{\Harm}{\mathop{\mathrm{Harm}}\nolimits}

\newcommand{\R}{\mathbb{R}}

\newcommand{\x}{{\bold x}}
\newcommand{\y}{{\bold y}}

\newcounter{mycounter}

\usepackage[dvipdfm,
bookmarks=true,bookmarksnumbered=true,
bookmarksopen=true]{hyperref}
\begin{document}

\title{Nonexistence of tight spherical design of harmonic index $4$}
\author{Takayuki Okuda and Wei-Hsuan Yu}
\address{
Department of Mathematics, Graduate School of Science, Hiroshima University
1-3-1 Kagamiyama, Higashi-Hiroshima, 739-8526 Japan}
\email{okudatak@hiroshima-u.ac.jp}
\address{Department of Mathematics, Michigan State University,
619 Red Cedar Road, East Lansing, MI 48824}
\email{u690604@gmail.com}
\thanks{The first author is supported by
Grant-in-Aid for JSPS Fellow No.25-6095 and the second author is supported in part by
NSF grants CCF1217894,  DMS1101697}
\date{}
\maketitle

\begin{abstract}
We give a new upper bound of the cardinality of a set of equiangular lines in $\R^n$
with a fixed angle $\theta$
for each $(n,\theta)$ satisfying certain conditions.
Our techniques are based on semi-definite programming methods for spherical codes  introduced by Bachoc--Vallentin [J.~Amer.~Math.~Soc.~2008].
As a corollary to our bound,
we show the nonexistence of spherical tight designs
of harmonic index $4$ on $S^{n-1}$ with $n \geq 3$.
\end{abstract}

\section{Introduction}

The purpose of this paper is to give a new upper bound of the cardinality of a set of equiangular lines with certain angles (see Theorem \ref{thm:rel}).
As a corollary to our bound, we show the nonexistence of tight designs of harmonic index $4$ on $S^{n-1}$ with $n \geq 3$ (see Theorem \ref{thm:nonex-tight}).

Throughout this paper,
$S^{n-1} := \{ x \in \R^n \mid \| x \| = 1 \}$
denotes the unit sphere in $\R^{n}$.
By Bannai--Okuda--Tagami \cite{ban13},
a finite subset $X$ of $S^{n-1}$ is called
\emph{a spherical design of harmonic index $t$ on $S^{n-1}$}
(or shortly, \emph{a harmonic index $t$-design on $S^{n-1}$})
if $\sum_{\x \in X} f(\x) = 0$
for any harmonic polynomial function $f$ on $\R^{n}$ of degree $t$.

Our concern in this paper is in
tight harmonic index $4$-designs.
A harmonic index $t$-design $X$ is said to be \emph{tight} if $X$
attains the lower bound given by \cite[Theorem 1.2]{ban13}.
In particular,
for $t = 4$,
a harmonic index $4$-design on $S^{n-1}$ is tight
if its cardinality is $(n+1)(n+2)/6$.
For the cases where $n=2$, we can construct tight harmonic index
$4$-designs as two points $\x$ and $\y$ on $S^{1}$ with the
inner-product $\langle \x,\y \rangle_{\R^2} = \pm \sqrt{1/2}$.

The paper \cite[Theorem 4.2]{ban13} showed that
if tight harmonic index $4$-designs on $S^{n-1}$ exist,
then $n$ must be $2$ or $3(2k-1)^2 -4$ for some integers $k \geq 3$.
As a main result of this paper,
we show that the later cases do not occur.
That is, the following theorem holds:
\begin{thm}\label{thm:nonex-tight}
For each $n \geq 3$,
spherical tight design of harmonic index $4$ on $S^{n-1}$ does not exist.
\end{thm}

A set of lines in $\R^n$ is called \emph{an equiangular line system} if
the angle between each pair of lines is constant.
By definition, an equiangular line system can be considered as
a spherical two-distance set with the inner product set $\{ \pm \cos \theta \}$
for some constant $\theta$.
Such the constant $\theta$ is called \emph{the common angle} of the equiangular line system.
The recent development of this topic can be found in \cite{barg14, grea14}.

By \cite[Proposition 4.2]{ban13},
any tight harmonic index $4$-design on $S^{n-1}$
can be considered as an equiangular line system with the common angle $\arccos \sqrt{3/(n+4)}$.
The proof of Theorem \ref{thm:nonex-tight}
will be reduced to a new relative upper bound (see Theorem \ref{thm:rel})
for the cardinalities of equiangular line systems with a fixed common angle.
Note that in some cases, our relative bound is better
than the Lemmens--Seidel relative bound
(see Section \ref{sec:main} for more details).

The paper is organized as follows:
In Section \ref{sec:main},
as a main theorem of this paper,
we give a new relative bound for the cardinalities of equiangular line systems with a fixed common angle satisfying certain conditions.
Theorem \ref{thm:nonex-tight} is followed as a corollary to our relative bound.
In Section \ref{sec:proof}, our relative bound is proved based on the method by Bachoc--Vallentin \cite{bac08a}.

\section{Main results}\label{sec:main}

In this paper, we denote by $M(n)$ and $M_{\cos \theta}(n)$ the maximum number of equiangular lines in $\R^n$ and that with the fixed common angle $\theta$, respectively.
By definition, \[
M(n) = \sup_{0 \leq \alpha < 1} M_\alpha(n).
\]
The important problems for equiangular lines are to give upper and lower estimates $M(n)$ or $M_\alpha(n)$ for fixed $\alpha$.
One can find a summary of recent progress of this topic in \cite{barg14, grea14}.

Let us fix $0 \leq \alpha < 1$.
Then for a finite subset $X$ of $S^{n-1}$ with $I(X) \subset \{ \pm \alpha \}$,
we can easily find an equiangular line system with the common angle $\arccos \alpha$
and the cardinality $|X|$,
where \[
I(X) := \{ \langle \x,\y \rangle_{\R^n} \mid \x,\y \in X \text{ with } \x \neq \y \}
\]
is the set of inner-product values of distinct vectors in $X \subset S^{n-1} \subset \R^n$.
The converse is also true.
In particular, we have
\begin{align*}
M_{\alpha}(n) = \max \{ |X| \mid X \subset S^{n-1} \text{ with } I(X) \subset \{ \pm \alpha \} \},
\end{align*}
and therefore, our problem can be considered as a problem in special kinds of spherical two-distance sets.

In this paper, we are interested in upper estimates of $M_\alpha(n)$.
According to \cite{lem73},
Gerzon gave the upper bound on $M(n)$ as
$M(n) \leq n(n+1)/2$ and therefore,
we have \[
M_\alpha(n) \leq \frac{n(n+1)}{2} \quad \text{ for any } \alpha.
\]
This upper bound is called the Gerzon absolute bound.
Lemmens and Seidel \cite{lem73} showed that
  \begin{equation*}
    M_\alpha(n) \le \frac {n(1-\alpha^2)}{1-n\alpha^2} \quad \text{ in the cases where } 1-n\alpha^2 > 0.
  \end{equation*}
This inequality is sometimes called the Lemmens--Seidel
 relative bound as opposed to the Gerzon absolute bound.

As a main theorem of this paper,
we give other upper estimates of $M_\alpha(n)$ as follows:

\begin{thm}\label{thm:rel}
Let us take $n \geq 3$ and $\alpha \in (0,1)$ with
\[
2-\frac{6\alpha-3}{\alpha^2} < n < 2 + \frac{6\alpha+3}{\alpha^2}.
\]
Then
\[
M_\alpha(n) \leq 2 + \frac{(n-2)}{\alpha} \max \left\{ \frac{(1-\alpha)^3}{(n-2)\alpha^2 +6\alpha-3}, \frac{(1+\alpha)^3}{-(n-2)\alpha^2 + 6\alpha+3} \right\}.
\]
In particular, for an integer $l \geq 2$, if
\[
3l^2-6l+2 < n < 3l^2 + 6l+2
\]
then
\[
M_{1/l}(n) \leq 2 + (n-2) \max \left\{ \frac{(l-1)^3}{-3l^2 + 6l + (n-2)}, \frac{(l+1)^3}{3l^2 + 6l -(n-2)} \right\}.
\]
\end{thm}

Recall that by \cite[Proposition 4.2, Theorem 4.2]{ban13} for $n \geq 3$,
if there exists a tight harmonic index $4$-design $X$ on $S^{n-1}$,
then $n = 3(2k-1)^2-4$ for some $k \geq 3$
and
\begin{align*}
M_{\sqrt{3/(n+4)}}(n) = (n+1)(n+2)/6.
\end{align*}

However, as a corollary to Theorem \ref{thm:rel},
we have the following upper bound of $M_{\sqrt{3/(n+4)}}(n)$
 and obtain Theorem \ref{thm:nonex-tight}.

\begin{cor}\label{cor:tight4-eq}
Let us put
$n_k := 3(2k-1)^2-4$ and $\alpha_k := \sqrt{3/(n_k+4)} = 1/(2k-1)$.
Then for each integer $k \geq 2$,
\begin{align*}
M_{\alpha_k}(n_k)
&\leq 2 (k-1) (4k^3-k-1) (< (n_k+1)(n_k+2)/6).
\end{align*}
\end{cor}

It should be remarked that in the setting of Corollary \ref{cor:tight4-eq},
the Lemmens--Seidel relative bound does not work since \[
1-n_k \alpha_k^2 = -2(4 k^2-4k-1)/(2k-1)^2 < 0
\]
and our bound is better than the Gerzon absolute bound.

The proof of Theorem \ref{thm:rel} is given in Section \ref{sec:proof}
based on Bachoc--Vallentin's SDP method for spherical codes \cite{bac08a}.
The origins of applications of the linear programming method in coding theory can be traced back to the work of Delsarte \cite{del73}.
Applications of semidefinite programming (SDP) method in coding theory and distance geometry gained momentum after the pioneering work of Schrijver \cite{Schrijver05code}
that derived SDP bounds on codes in the Hamming and Johnson spaces.
Schrijver's approach was based on the so-called Terwilliger algebra
of the association scheme.
The similar idea for spherical codes were formulated by Bachoc and Vallentin \cite{bac08a}
regarding for kissing number problems.
Barg and Yu \cite{barg13} modified it to achieve maximum size of spherical two-distance sets in $\mathbb{R}^n$ for most values of $n \leq 93$.
In our proof, we restricted the method to obtain upper bounds for equiangular line sets.

We can see in \cite{GST06upperbounds,Schrijver05code} and \cite{bac08a, bac09opti, bac09sdp, Musin08bounds}
that the SDP method works well
for studying binary codes and spherical codes, respectively.
Especially, for equiangular lines, Barg and Yu \cite{barg14} give the best known upper bounds of $M(n)$ for some $n$ with $n \leq 136$ by the SDP method.
Our bounds in Corollary \ref{cor:tight4-eq} are the same as \cite{barg14} in lower dimensions.
However, in some cases, we need some softwares to complete the SDP method.
It should be emphasized that our theorem offer upper bound of $M_{\alpha_k}(n_k)$ for arbitrary large $n_k$
and the proof can be followed by hand calculations without using any convex optimization software.

\section{Proof of our relative bound}\label{sec:proof}

To prove Theorem \ref{thm:rel},
we apply Bachoc--Vallentin's SDP method for spherical codes in \cite{bac08a} to spherical two-distance sets.
The explicit statement of it was given by Barg--Yu \cite{barg13}.

We use symbols $P^{n}_l(u)$ and $S^n_l(u,v,t)$
in the sense of \cite{bac08a}.
It should be noted that
the definition of $S^{n}_l(u,v,t)$ is different from
that of \cite{bac09opti} and \cite{barg13}
(see also \cite[Remark 3.4]{bac08a} for such the differences).

In order to state it,
we define
\begin{align*}
W(x)&:= \begin{pmatrix}1&0\\0&0\end{pmatrix} +
\begin{pmatrix}0&1\\1&1\end{pmatrix} (x_1+x_2)/3 +
\begin{pmatrix}
0&0\\0&1
\end{pmatrix} (x_3+x_4+x_5+x_6), \\
S^{n}_l(x;\alpha,\beta) &:= S^{n}_{l}(1,1,1)+S^{n}_l(\alpha,\alpha,1)x_1 + S^{n}_l(\beta,\beta,1) x_2 + S^{n}_l(\alpha,\alpha,\alpha) x_3 \\
& \quad \quad + S^{n}_l(\alpha,\alpha,\beta) x_4 + S^{n}_l(\alpha,\beta,\beta) x_5 + S^{n}_l(\beta,\beta,\beta) x_6
\end{align*}
for each $x = (x_1,x_2,x_3,x_4,x_5,x_6) \in \R^6$ and $\alpha, \beta \in [-1,1)$.
We remark that $W(x)$ is a symmetric matrix of size $2$ and
$S^{n}_{l}(x;\alpha,\beta)$ is a symmetric matrix of infinite size
indexed by $\{ (i,j) \mid i,j = 0,1,2,\dots, \}$.

\begin{fact}[Bachoc--Vallentin \cite{bac08a} and Barg--Yu \cite{barg13}]\label{fact:SDP-problem}
Let us fix $\alpha, \beta \in [-1,1)$.
Then
any finite subset $X$ of $S^{n-1}$ with $I(X) \subset \{ \alpha,\beta \}$
satisfies
\[
|X| \leq \max \{ 1 + (x_1 + x_2)/3 \mid x = (x_1,\dots,x_6) \in \Omega^{n}_{\alpha,\beta} \}
\]
where the subset $\Omega^{n}_{\alpha,\beta}$ of $\R^{6}$ is defined by
\[
\Omega_{\alpha,\beta}^{n} := \{\, x = (x_1,\dots,x_6) \in \R^{6} \mid \text{ $x$ satisfies the following four conditions } \,\}.
\]
\begin{enumerate}
\item $x_i \geq 0$ for each $i = 1,\dots,6$.
\item $W(x)$ is positive semi-definite.
\item $3 + P^{n}_l(\alpha) x_1 + P^{n}_l (\beta) x_2 \geq 0$ for each $l=1,2,\dots.$
\item Any finite principal minor of $S^{n}_l(x;\alpha,\beta)$ is positive semi-definite for each $l = 0,1,2,\dots.$
\end{enumerate}
\end{fact}

To prove Theorem \ref{thm:rel},
we use the following ``linear version'' of Fact \ref{fact:SDP-problem}:

\begin{cor}\label{cor:triangleLP}
In the same setting of Fact $\ref{fact:SDP-problem}$,
\[
|X| \leq \max \{ 1 + (x_1 + x_2)/3 \mid x = (x_1,\dots,x_6) \in \widetilde{\Omega}^{n}_{\alpha,\beta} \}
\]
where the subset $\widetilde{\Omega}^{n}_{\alpha,\beta}$ of $\R^{6}$ is defined by
\[
\widetilde{\Omega}_{\alpha,\beta}^{n} := \{\, x = (x_1,\dots,x_6) \in \R^{6} \mid \text{ $x$ satisfies the following three conditions } \,\}.
\]
\begin{enumerate}
\item $x_i \geq 0$ for each $i = 1,\dots,6$.
\item $\det W(x) \geq 0$.
\item $(S^{n}_l)_{i,i}(x;\alpha,\beta) \geq 0$ for each $l,i = 0,1,2,\dots$,
where $(S^{n}_l)_{i,i}(x;\alpha,\beta)$ is the $(i,i)$-entry of the matrix $S^{n}_l(x;\alpha,\beta)$.
\end{enumerate}
\end{cor}

By Corollary \ref{cor:triangleLP},
the proof of Theorem \ref{thm:rel} can be reduced to show the following proposition:

\begin{prop}\label{prop:SDPrel}
Let $n \geq 3$ and $0 < \alpha <1$.
Then the following holds:
\begin{enumerate}
\item
\begin{align*}
\max \{ 1+(x_1+x_2)/3 \mid x \in \widetilde{\Omega}^{n}_{\alpha,-\alpha} \} \leq 2 + 2 + (n-2)\frac{(1-\alpha)^3}{\alpha((n-2)\alpha^2 +6\alpha-3)}
\end{align*}
if
$(1-\alpha)^3(-(n-2)\alpha^2 + 6\alpha+3) \geq
(1+\alpha)^3((n-2)\alpha^2 +6\alpha-3) \geq 0$.
\item
\begin{align*}
\max \{ 1+(x_1+x_2)/3 \mid x \in \widetilde{\Omega}^{n}_{\alpha,-\alpha} \} \leq 2 + (n-2) \frac{(1+\alpha)^3}{\alpha(-(n-2)\alpha^2 + 6\alpha+3)}
\end{align*}
if
$(1+\alpha)^3((n-2)\alpha^2 +6\alpha-3) \geq (1-\alpha)^3(-(n-2)\alpha^2 + 6\alpha+3) \geq 0$.
\end{enumerate}
\end{prop}

For the proof of Proposition \ref{prop:SDPrel}, we need the next
explicit formula of $(S^{n}_3)_{1,1}$ which are obtained by direct
computations:

\begin{lem}\label{lem:S3explicite}
For each $-1 < \alpha < 1$,
\begin{align*}
(S^{n}_3)_{1,1}(1,1,1) &= 0 \\
(S^{n}_3)_{1,1}(\alpha,\alpha,1) &= \frac{n(n+2)(n+4)(n+6)}{3(n-1)(n+1)(n+3)}
\alpha^2 (1-\alpha^2)^3 \\
(S^{n}_3)_{1,1}(\alpha,\alpha,\alpha) &=
-\frac{n(n+2)(n+4)(n+6)}{(n-2)(n-1)(n+1)(n+3)}
(\alpha-1)^3 \alpha^3 ((n-2)\alpha^2-6\alpha-3) \\
(S^{n}_3)_{1,1}(\alpha,\alpha,-\alpha) &=
-\frac{n(n+2)(n+4)(n+6)}{(n-2)(n-1)(n+1)(n+3)}
\alpha^3 (\alpha+1)^3 ((n-2)\alpha^2 +6\alpha-3).
\end{align*}
\end{lem}

\begin{proof}[Proof of Proposition $\ref{prop:SDPrel}$]
Fix $\alpha$ with $0 < \alpha < 1$ and
take any $x \in \widetilde{\Omega}^{n}_{\alpha,-\alpha}$.
For simplicity we put $X = (x_1+x_2)/3$, $Y = x_3+x_5$ and $Z = x_4+x_6$.
By computing $\det W(x)$,
we have
\begin{align}
-X(X-1)+Y+Z \geq 0. \label{eq:Wrel}
\end{align}
Furthermore,
we have $(S^{n}_3)_{1,1}(x;\alpha,-\alpha) \geq 0$,
and hence, by Lemma \ref{lem:S3explicite},
\begin{multline*}
(n-2) \frac{(1-\alpha^2)^3}{\alpha} X
- (1-\alpha)^3 (-(n-2)\alpha^2 + 6\alpha+3) Y \\
- (1+\alpha)^3 ((n-2)\alpha^2 +6\alpha-3) Z \geq 0
\end{multline*}
Therefore,
in the cases where
\[
(1-\alpha)^3(-(n-2)\alpha^2 + 6\alpha+3) \geq
(1+\alpha)^3((n-2)\alpha^2 +6\alpha-3) \geq 0,
\]
we obtain
\begin{align*}
(n-2) \frac{(1-\alpha^2)^3}{\alpha} X - (1+\alpha)^3 ((n-2)\alpha^2 +6\alpha-3) (Y+Z) \geq 0.
\end{align*}
By combining with \eqref{eq:Wrel},
\begin{align*}
(n-2) \frac{(1-\alpha^2)^3}{\alpha} X - (1+\alpha)^3 ((n-2)\alpha^2 +6\alpha-3) X(X-1) \geq 0
\end{align*}
Thus we have
\[
2 + (n-2) \frac{(1-\alpha)^3}{\alpha((n-2)\alpha^2 +6\alpha-3)} \geq X+1 = 1 + (x_1+x_2)/3.
\]
By the similar arguments,
in the cases where
\[
(1+\alpha)^3((n-2)\alpha^2 +6\alpha-3) \geq
(1-\alpha)^3(-(n-2)\alpha^2 + 6\alpha+3) \geq 0,
\]
we have
\[
2 + (n-2) \frac{(1+\alpha)^3}{\alpha(-(n-2)\alpha^2 +6\alpha+3)} \geq X+1 = 1 + (x_1+x_2)/3.
\]
\end{proof}

\begin{rem}
Harmonic index $4$-designs are defined by using
the functional space $\Harm_{4}(S^{n-1})$.
Therefore, it seems to be natural to consider $\Harm_4(S^{n-1})$
in Bachoc--Vallentin's SDP method.
In our proof, the functional space
\[
H^{n-1}_{3,4} \subset \bigoplus_{m=0}^{4} H^{n-1}_{m,4} = \Harm_{4}(S^{n-1})
\]
$($see \cite{bac08a} for the notation of $H^{n-1}_{m,l}$$)$
plays an important role to show the nonexistence of tight designs of harmonic index $4$
since $(S^n_3)_{1,1}$ comes from $H^{n-1}_{3,4}$.
We checked that if we consider
$H^{n-1}_{0,4} \oplus H^{n-1}_{1,4} \oplus H^{n-1}_{2,4} \oplus
H^{n-1}_{4,4}$ instead of $H^{n-1}_{3,4}$,
our upper bound can not be obtained for small $k$.
However, we can not find any conceptional reason of the importance of $H^{n-1}_{3,4}$.
\end{rem}

\section*{Acknowledgements.}
The authors would like to give heartfelt thanks to Eiichi Bannai, Alexander Barg and Makoto Tagami whose suggestions and comments were of inestimable value for this paper.
The authors also would like to thanks
Akihiro Munemasa, Hajime Tanaka and Ferenc Sz{\"o}ll{\H o}si
for their valuable comments.

\providecommand{\bysame}{\leavevmode\hbox
to3em{\hrulefill}\thinspace} \providecommand{\href}[2]{#2}
\bibliographystyle{amsalpha}

\end{document}